\documentclass{amsart}

\usepackage{etex}
\usepackage{amsmath, amssymb}
\usepackage{array}
\usepackage{graphpap, color, paralist, pstricks}
\usepackage[mathscr]{eucal}
\usepackage[pdftex]{graphicx}
\usepackage[pdftex,colorlinks,backref=page,citecolor=blue]{hyperref}
\usepackage{pifont}

\newtheorem{thm}{Theorem}[section]
\newtheorem{prop}[thm]{Proposition}

\newtheorem{lem}[thm]{Lemma}

\theoremstyle{definition}

\newtheorem{remark}[thm]{Remark}

\newcommand{\gO}{\Omega}

\newcommand{\cO}{\mathcal{O} }

\newcommand{\cB}{\mathcal{B} }

\newcommand{\cE}{\mathcal{E} }

\newcommand{\cG}{\mathcal{G} }

\newcommand{\cW}{\mathcal{W} }

\newcommand{\beq}[1]{\begin{equation}\label{#1}}
\newcommand{\enq}[0]{\end{equation}}

\newcommand{\nin}[0]{\noindent}

\newcommand{\sub}[0]{\subseteq}

\newcommand{\bis}[0]{\mbox{\rm{bis}}}

\newcommand{\aaa}{q}

\begin{document}

\title{Note on the number of balanced independent sets in the Hamming cube}

\author{Jinyoung Park}
\thanks{The author is supported directly by NSF
grant DMS-1926686
and indirectly by NSF grant CCF-1900460.}
\email{jpark@math.ias.edu}
\address{School of Mathematics, Institute for Advanced Study \\
1 Einstein Drive, Princeton, NJ 08540, USA}

\begin{abstract}
Let $Q_d$ be the $d$-dimensional Hamming cube and $N=|V(Q_d)|=2^d$. An independent set $I$ in $Q_d$ is called balanced if $I$ contains the same number of even and odd vertices. We show that the logarithm of the number of balanced independent sets in $Q_d$ is
\[(1-\Theta(1/\sqrt d))N/2.\]
The key ingredient of the proof is an improved version of ``Sapozhenko's graph container lemma.''
\end{abstract}

\maketitle

\section{Introduction}
For a bipartite graph $G=X \coprod Y$ and an independent set $I$ in $G$, $I$ is said to be \textit{balanced} if $|I \cap X|=|I \cap Y|$. We use $\bis(G)$ for the number of balanced independent sets (BIS's) of a graph $G$.

Write $Q_d$ for the $d$-dimensional Hamming cube and $N$ for $|V(Q_d)|(=2^d)$. In this note we prove the following result on $\log \bis(Q_d)$. (All $\log$'s in this paper are in base 2.)

\begin{thm}\label{MT}
\beq{eq.MT} \log \bis(Q_d) =(1-\Theta(1/\sqrt d))N/2. \enq
\end{thm}

\nin It is easy to see that the rhs of \eqref{eq.MT} is a lower bound: Barber \cite{Barber} showed that the size of a maximum BIS in $Q_d$ is

\[\begin{cases}
\displaystyle 2^{d-1}-2{d-2 \choose (d-2)/2} &\mbox{if $d$ is even;}\bigskip\\
\displaystyle 2^{d-1}-{d-1 \choose (d-1)/2} &\mbox{if $d$ is odd,}
\end{cases}\]
and collecting balanced subsets of a maximum BIS gives the lower bound. So the main task of this paper is to show the rhs of \eqref{eq.MT} is also an upper bound.

\bigskip

\nin \textit{Background.} The asymptotics for the number of (ordinary) independent sets in $Q_d$, $i(Q_d)$, was first {given} by Korshunov and Sapozhenko \cite{KS}:

\begin{thm}\label{i}
\beq{i asymp} i(Q_d)\sim 2\sqrt e 2^{N/2}.\enq
\end{thm}

\nin (The above asymptotics are substantially refined by Jenssen and Perkins in \cite{JP}.) Note that the rhs of \eqref{i asymp} is an asymptotic lower bound on $i(Q_d)$: writing $Q_d=\cE \coprod \cO$ (a few basic definitions are recalled below), any subset of $\cE$ or $\cO$ is an independent set, from which we have { $ 2\cdot2^{N/2}-1$} independent sets. The extra factor $\sqrt e$ reflects the contribution of independent sets most of whose vertices are even (odd, resp.), together with a (very) small number of odd (even, resp.) vertices. (See e.g. \cite{GS} for a more detailed description on this lower bound construction.)

Thus Theorem \ref{i} {implies} that $i(Q_d)$ is asymptotically equal to this lower bound, and in particular, this implies all but a negligible fraction of independent sets in $Q_d$ are highly unbalanced. The natural problem of estimating $\bis(Q_d)$ was suggested by T. Helmuth, M. Jenssen, and W. Perkins \cite{Q}, and Theorem~\ref{MT} answers this question at the level of asymptotics of the logarithm.

The key ingredient of the proof of Theorem \ref{MT} is Lemma \ref{main lemma}, an improvement (see Remark \ref{rmk:imp}) of ``Sapozhenko's graph container lemma'' from \cite{Sap87}. Sapozhenko's lemma and its variants have played a key role in resolving a number of asymptotic enumeration problems on the Hamming cube and related structures, e.g. \cite{KS, JP, G, KPq, misqn, JK, B}. The current improved version of the lemma is implicitly proved in \cite[Lemma~6.3]{misqn}, but we give a self-contained proof in Section \ref{sec:ml} to provide a convenient reference for future work.

\bigskip

\nin \textit{Definitions.} We use $Q_d$ for the $d$-dimensional Hamming cube: that is, $V=V(Q_d)$ is the collection of binary strings of length $d$, and two vertices are adjacent iff they differ in exactly one coordinate. A vertex $v$ is even (odd, resp.) if $v$ contains an even (odd, resp.) number of $1$'s. We use $\cE$ ($\cO$, resp.) for the set of even (odd, resp.) vertices in $Q_d$ (so $Q_d=\cE \coprod \cO$). The collection of balanced independent sets (BIS's) in $Q_d$ is denoted by {$\cB=\cB(d)$}, and $I$ always denotes a BIS.

As usual, $N(v)$ is the set of neighbors of $v$, and $N(A)$ is the set of vertices that are adjacent to at least one vertex in $A$. We use $[A]$ for the \textit{closure} of $A$, namely, $[A]=\{v \in V:N(v) \subseteq N(A)\}$.

Finally, we {refer to the logarithm of the number of possibilities for a choice as the \textit{cost} of that choice.}

\bigskip

\nin \textit{Outline.} In Section \ref{sec:tools} we recall some basic tools. The main lemma (Lemma \ref{main lemma}) and Theorem \ref{MT} are proved in Section \ref{sec:ml} and Section \ref{sec:MT} respectively.

\section{Tools}\label{sec:tools}

The following is a well-known fact about the sum of binomial coefficients.

\begin{prop}\label{binom.sum}
For any fixed $\alpha \in [0, 1/2]$ and $n \in \mathbb Z^+$,
\[\sum_{i \le \alpha n}{n \choose i} \le 2^{H(\alpha)n},\]
where $H(\alpha):=-\alpha \log \alpha -(1-\alpha)\log (1-\alpha)$ is the binary entropy {function}.
\end{prop}

For a positive integer $m$, a \textit{composition} of $m$ is a sequence $(a_1, \ldots, a_s)$ of positive integers summing to $m$. {Recall the following basic fact:}

\begin{prop}\label{prop:comp} 
The number of compositions of $m$ is $2^{m-1}$ and the number with at most $b \le m/2$ parts is
\[ \sum_{i\le b}{m-1 \choose i} {\le} \exp_2[b\log (em/b)].\]
\end{prop}

{S}ay $A \subseteq V$ is \textit{{$2$}-linked} if for any $u, v \in A$, there are vertices $u=u_0, u_1, \ldots, u_l=v$ in $A$ such that for each $i \in [l]$, $u_{i-1}$ and $u_i$ are at distance at most $2$ in $Q_d$. The \textit{$2$-components} of $A$ are its maximal $2$-linked subsets.

\begin{prop} [\cite{G}, Lemma 1.6] \label{prop:setcost} 
For each fixed $k$, the number of $k$-linked subsets of $V$ of size $x$ containing some specified vertex is at most $2^{O(x\log d)}$.
\end{prop}

The next two results recall standardish isoperimetric inequalities for $Q_d$. Recall that $N=|V(Q_d)|=2^d$.

\begin{prop}[\cite{GS}, Claim 2.5] \label{prop:isop}
For $A \subseteq \cE$ (or $\cO$) with $|A| \le N/4$,
\[\frac{|N(A)|-|A|}{|N(A)|}=\gO(1/\sqrt d).\]
\end{prop}

\begin{prop}[\cite{GS}, Lemma 2.6]\label{prop:isop1} 
For $A \subseteq \cE$ (or $\cO$),
\[\mbox{if } |A|<d^{O(1)}, \mbox{ then } |A| \le O(1/d)|N(A)|.\]
\end{prop}

{The next lemma recalls what we need from \cite{Sap87}, and follows from Lemmas 5.3-5.5 in the excellent exposition due to Galvin \cite{GS}.} For $A$ in the statement, we use $G=N(A)$ and $t=|G|-|[A]|$. 

\begin{lem} \label{lem:GS}
For $\aaa, g \in \mathbb Z^+$, $\aaa \le N/4$, $g \ge d^4$ and
\[\cG(\aaa,g)=\{A \subseteq \cE: \mbox{$A$ is 2-linked, $|[A]|=\aaa$ and $|G|=g$}\},\]
there {exist a family} $\cW=\cW(\aaa, g) \subseteq 2^{\cE} \times 2^{\cO}$ with
\beq{cW} |\cW|=2^{O(t\log^2d/\sqrt d)}\enq
and {a function} $\Phi=\Phi_{\aaa,g}:\cG \rightarrow \cW$ such that for each $A \in \cG$, $(S,F):=\Phi(A)$ satisfies:

\begin{enumerate}[(a)]
\item $S \supseteq [A], F \subseteq G$;
\item $|S| \le |F| + O(t/(\sqrt d \log d)).$
\end{enumerate}
\end{lem}

\section{Main Lemma}\label{sec:ml}

{In this section we prove the following key lemma.}

\begin{lem}\label{main lemma}

For $\aaa, g,$ and $ \cG(\aaa, g)$ as in Lemma \ref{lem:GS},
\[\log|\cG(\aaa,g)|\le g-\gO(t).\]
\end{lem}

\begin{remark}\label{rmk:imp}
For comparison, Sapozhenko's original graph container lemma says
\beq{Sap.original} \log|\cG(\aaa,g)|\le g-\gO(t/\log d),\enq
so the main contribution of Lemma \ref{main lemma} is to improve the $\gO(t/\log d)$-term in the rhs of \eqref{Sap.original}  to $\gO(t)$. This improvement plays a crucial role in the current work: the bound in \eqref{Sap.original} would give a weaker bound, $2^{g-\gO(N/(\sqrt d \log d))}$, in \eqref{main task}. 

\end{remark}

\begin{proof}[Proof of Lemma \ref{main lemma}]
Given $\aaa$ and $g$, Lemma \ref{lem:GS} gives $\cW=\cW(\aaa,g)$ at cost $O(t \log^2 d/\sqrt d)$. So it suffices to show that given $(S,F) \in \cW$, the cost of specifying $A \in \Phi^{-1}(S,F)$ is at most $g-\gO(t)$.

Let $\gamma \in (0,1)$ be a constant TBD. (We don't try to optimize $\gamma$.)

Case 1. If $|S| <g-\gamma t$, then we specify $A$ by picking a subset of $S$, which costs $|S|=g-\gamma t$.

Case 2. If $|S| \ge g-\gamma t$, then we first fix an arbitrary closed $A^* \in \Phi^{-1}(S,F)$ (we can simply pick \textit{any} member of $\Phi^{-1}(S,F)$ and take its closure). Note that this choice is free, and $|A^*|=\aaa$ by the definition of $\cG(\aaa,g)$.

The crucial observation is that ({letting} $G^*=N(A^*)$)
\[\mbox{$(G^* \setminus G, G \setminus G^*)$ determines $(G,[A])$.}\]
In what follows we first specify $G^* \setminus G$ and $G \setminus G^*$ from which we have $[A]$, and then specify $A \subseteq [A]$.

We first bound the cost of $G^* \setminus G$. Since $G^* \setminus G \subseteq G^* \setminus F$, the cost of $G^* \setminus G$ is at most (using Lemma~\ref{lem:GS}~(b)) 
\beq{l1} |G^* \setminus F|=|G^*| -|F| \le |G|-|S|+O(t/(\sqrt d \log d))\le (1+o(1))\gamma t.\enq

Next, we bound the cost of $G \setminus G^*$. Observe that
\[G \setminus G^* =N([A] \setminus A^*) \setminus G^*,\]
because each $x \in G \setminus G^*$ has a neighbor in $[A]$ and none in $A^*$. So we may specify $G \setminus G^*$ by specifying a $Y \subseteq [A] \setminus A^* \subseteq S \setminus A^*$ with $G \setminus G^*=N(Y) \setminus G^*$. Moreover, we only need $Y \sub S \setminus A^*$ of size at most $|G \setminus G^*| \le g-|F| \le (1+o(1))\gamma t$, by letting $Y$ contain one neighbor of $x$ for each $x \in G \setminus G^*$.

 Now, since (again using Lemma \ref{lem:GS}~(b)) 
\[|S \setminus A^*|=|S|-\aaa \le |F|+o(t)-\aaa \le g-\aaa +o(t) \le (1+o(1))t,\]
 the cost of specifying $Y$ from $S \setminus A^*$ is at most
\beq{l2}\log {(1+o(1))t \choose (1+o(1))\gamma t} \le (1+o(1))H(\gamma) t\enq
where $H(\cdot)$ is the binary entropy {function}.
Finally, once we have $[A]$, we specify $A$ by picking a subset of $[A]$, which costs
\beq{l3} \aaa=g-t. \enq
Summing up \eqref{l1}, \eqref{l2}, and \eqref{l3}, we bound the total cost for Case 2 by
\beq{Case2} (1+o(1))\gamma t + (1+o(1))H(\gamma) t+(g-t).\enq
Now, choose $\gamma$ so that \eqref{Case2} is less than (say) $g-t/2$, and the lemma follows.
\end{proof}

\section{Proof of Theorem \ref{MT}}\label{sec:MT}

We show that the rhs of \eqref{eq.MT} is an upper bound on $\log \bis(Q_d)$. We first dispose of the minor cost for small { independent sets.}

\begin{prop}\label{tiny} There is a constant $\alpha \in (0,1/2)$ such that
\beq{eq:tiny}|\{I \in {\cB}:|I|\le \alpha N\}| =2^{(1-\gO(1))N/2}.\enq

\end{prop}

\begin{proof}
The lhs of \eqref{eq:tiny} is at most (with plenty of room)
\[\left[ \sum_{0 \le k \le \alpha N/2} {N/2 \choose k} \right]^2 \le 2^{H(\alpha)N}\]
(the inequality uses Proposition \ref{binom.sum}), and the rhs is less than $2^{(1-\gO(1))N/2}$ for small enough constant $\alpha$.\end{proof}

Let {$\cB'=\{I \in \cB: |I|>\alpha N\}$} where $\alpha$ is the constant in Proposition \ref{tiny}. A natural way to specify {a balanced independent set} $I$ is to choose a set $A \sub \cE$ and a set $B \sub \cO \setminus N(A)$ so that $|A|=|B|$ (and take $I=A \cup B$). Moreover, since $[I \cap \cE]$ and $[I \cap \cO]$ have no edges between them, $A$ and $B$ must satisfy $\min\{|[A]|, |[B]|\} \le N/4$ (because $|N(X)|\ge|X|, \; \forall X \sub \cE \mbox{ or } \cO$). Thus, {$|\cB'|$} is at most
\[\begin{split} &2\times \sum_{g >\alpha N/2} \sum_{\substack{A \sub \cE: |N(A)|=g\\ |A| \ge \alpha N/2 \\ |[A]| \le N/4}} |\{B \sub \cO \setminus N(A)\}|\\ &= 2^{N/2+1} \sum_{g >\alpha N/2} 2^{-g} |\{A \sub \cE: |N(A)|=g, |A| \ge \alpha N/2, |[A]| \le N/4\}|.\end{split} \]

Our main task is to show that
\beq{main task}\mbox{given $g =\gO(N)$, $|\{A \sub \cE: |N(A)|=g, |A| \ge \alpha N/2, |[A]|\le N/4\}| \le 2^{g-\gO(N/\sqrt d)}$,}\enq
from which it follows that (with Proposition \ref{tiny})
\[{|\cB|} \le 2^{(1-\gO(1))N/2}+2^{N/2+1}\sum_{g>\alpha N/2} 2^{-\gO(N/\sqrt d)}=2^{(1-\gO(1/\sqrt d))N/2}.\]
In the rest of the paper, we show \eqref{main task}. In what follows, we always assume that $g$ and $A$ satisfy the restrictions in \eqref{main task}.

\bigskip

\nin \textit{Notation.}

Recall that a $2$-component of $A$ is a maximal $2$-linked subset of $A$ (see Section \ref{sec:tools}).

\begin{itemize}
\item $A_i$'s: 2-components of $A$.

\item $G_i=N(A_i)$, $G=\cup_i G_i=N(A)$.

\item $g_i=|G_i|$, $a_i=|A_i|$, $\aaa_i=|[A_i]|$, $t_i=g_i-\aaa_i$.

\item $c(A)=\sum_i \aaa_i$ (note that $|A| \le c(A) \le |[A]|$).

\item $g=|G|~(=\sum_i g_i)$.

\item $A_i \mbox{ (or simply $i$) is }
\begin{cases} \mbox{\textit{ isolated } if $a_i=1$ (equiv. $g_i=d$);} \\ \mbox{\textit{ small } if $A_i$ is not isolated and $g_i<d^4$;} \\
\mbox{\textit{ large } otherwise}. \end{cases}$
\end{itemize}
\nin Note that the classification in the above bullet point is entirely determined by $g_i$.

By Proposition \ref{prop:isop1},
\[\sum\{\mbox{$a_i$: $i$ isolated or small}\}=O(N/d),\]
so in particular, we have (since $|A|=\gO(N)$)
\beq{i large}\sum\{\mbox{$g_i$: $i$ large}\}>\sum\{\mbox{$a_i$: $i$ large}\}=\gO(N).\enq

\begin{proof}[Proof of \eqref{main task}]
Observe that (since $|A| \le c(A) \le |[A]|$) it suffices to show that given $g$ as in \eqref{main task} and $\aaa$ with $\alpha N/2 \le \aaa \le \min\{N/4, g\}$, the number of $A$'s in $\cE$ with $c(A)=\aaa$ and $|N(A)|=g$ is at most $2^{g-\gO(N/\sqrt d)}$ ({since t}hen summing this up over all $\aaa$'s gives \eqref{main task}).

Given $\aaa$ and $g$, we first decompose $(\aaa,g)$ into $\{(\aaa_i,g_i)'s\}$ so that $\sum_i \aaa_i=\aaa$ and $\sum_i g_i=g$ (and then specify $A_i$'s satisfying $|[A_i]|=\aaa_i$ and $|G_i|=g_i$). The number of elements in a decomposition $\{(\aaa_i,g_i)'s\}$ is at most $g/d$, so Proposition \ref{prop:comp} bounds the cost of the $g_i$'s by $(g/d) \log (ed)$ and that of the $\aaa_i$'s by
\[
\left\{\begin{array}{ll}
(g/d)\log (ed)&\mbox{if $(g>)~ \aaa > 2g/d$;}\\
2g/d&\mbox{if $\aaa \le 2g/d$.}
\end{array}\right.
\]
Therefore, the total cost of the specification of $\aaa_i$'s and $g_i$'s is at most
\beq{pf:decomp}
O(g\log d/d).
\enq

\begin{lem}\label{small A}
Given $(\aaa_i,g_i)$, if $i$ is isolated or small, then the cost of $A_i$ with $|[A_i]|=\aaa_i$ and $|G_i|=g_i$ is at most $g_i$.
\end{lem}

\begin{proof}
The cost of an isolated $i$ is at most
\[\log(N/2)=d-1\le g_i.\]

For a small $i$, we use Proposition \ref{prop:setcost} to bound the cost of $[A_i]$ by 
\[\log(N/2)+O(\aaa_i \log d)\]
($\log(N/2)$ is the cost for the "specified vertex" in Proposition \ref{prop:setcost}). Once we have $[A_i]$ we specify $A_i$ by choosing each subset of $[A_i]$, which costs $\aaa_i$. Therefore, the total cost for small $i$'s is
\[\log(N/2)+O(\aaa_i\log d)+\aaa_i \le g_i,\]
 where the inequality follows from the fact that $g_i /2 \ge d-1$ and Proposition \ref{prop:isop1}. \end{proof}

Finally, given $(\aaa_i,g_i)$ {such that} $i$ is large, Lemma~\ref{main lemma} bounds the cost for $A_i$ with $|[A_i]|=\aaa_i$ and $|G_i|=g_i$ by
\beq{large A}g_i-\gO(t_i)\enq
(here we need the assumption that $(\aaa_i \le )~ \aaa \le N/4$ to apply Lemma \ref{main lemma}).

Summing up the costs in \eqref{pf:decomp}, Lemma \ref{small A} and \eqref{large A}, we have the cost for $A$ at most
\beq{total} O(g\log d/d)+g-\sum\{\mbox{$\gO(t_i)$: $i$ large}\}.\enq
Now Proposition \ref{prop:isop} gives $t_i=\gO(g_i/\sqrt d)$ for all $i$, so by \eqref{i large} we bound \eqref{total} by
\beq{total1}g-\gO(N/\sqrt d).\enq
\end{proof}

\nin \textbf{Acknowledgment.} The author is grateful to Matthew Jenssen for insightful conversations and helpful comments on the first draft of this paper.


\end{document}